\documentclass[a4paper]{amsart}
\usepackage{lmodern}
\usepackage{palatino}                       
\usepackage[bigdelims,vvarbb]{newpxmath}    
\usepackage[scaled=0.95]{inconsolata}       
\usepackage[T1]{fontenc}                    
\usepackage{comment}

\usepackage[utf8]{inputenc}
\usepackage{amsmath}
\usepackage{amsthm}
\usepackage{tikz}
\usepackage{graphicx}
\usepackage{float}
\usepackage{svg}
\usepackage{verbatim}
\usepackage{url}
\usepackage{mathtools}
\usepackage{microtype}
\usepackage{thm-restate}
\usepackage{enumitem}
\usepackage[hidelinks]{hyperref}
\usepackage[nameinlink]{cleveref}
\usepackage{caption}
\let\cref\Cref
\Crefname{enumi}{}{}

\theoremstyle{plain}
\newtheorem{theorem}{Theorem}[section]
\theoremstyle{definition}
\newtheorem{definition}[theorem]{Definition}
\theoremstyle{plain}
\newtheorem{lemma}[theorem]{Lemma}
\newtheorem{question}[theorem]{Question}
\theoremstyle{definition}
\newtheorem{example}[theorem]{Example}

\newenvironment{Example}{\begin{example}\rm}{\end{example}}
\theoremstyle{plain}
\newtheorem{proposition}[theorem]{Proposition}

\theoremstyle{plain}

\theoremstyle{remark}
\newtheorem{remark}[theorem]{Remark}

\newcommand{\qua}{\hskip 0.4em \ignorespaces}
\def\arxiv#1{\relax\ifhmode\unskip\qua\fi
\href{http://arxiv.org/abs/#1}%
{\tt arXiv:\penalty -100\unskip#1}}

\def\MR#1{\relax\ifhmode\unskip\qua\fi
\href{https://mathscinet.ams.org/mathscinet-getitem?mr=#1}{\tt MR#1}}
\def\ZB#1{\relax\ifhmode\unskip\qua\fi
\href{https://zbmath.org/?q=an:#1}{\tt Zbl\:#1}}
\def\xox#1{\csname xx#1\endcsname}

\renewenvironment{thebibliography}[1]{
  \begin{oldthebibliography}{#1}\small
    \setlength{\itemsep}{.5ex}
    \setlength{\parskip}{0em}
}
{
  \end{oldthebibliography}
}

\let\stdthebibliography\thebibliography
\let\stdendthebibliography\endthebibliography
\renewenvironment*{thebibliography}[1]{%
  \stdthebibliography{MM20.}}
  {\stdendthebibliography}

\newcommand{\myemail}[1]{\href{mailto:#1}{#1}}

\title{Non-fibered strongly quasipositive links and tightness}
\address{University of Glasgow, 132 University Pl, Glasgow G12 8TA, UK\medskip}
\author{Isacco Nonino}
\email{\myemail{2754452n@student.gla.ac.uk}}
\urladdr{\url{https://sites.google.com/view/isaccononinomath/home-page}}
\address{Section de mathématiques, Université de Genève, Suisse\medskip}
\author{Miguel {Orbegozo Rodriguez}}
\email{\myemail{miguel.orbegozorodriguez@gmail.com}}
\urladdr{\url{https://sites.google.com/view/miguel-orbegozo-rodriguez/home}}

\begin{document}
\begin{abstract}
    It is well known that for fibered links in $\mathbb{S}^3$ being strongly quasipositive and supporting a tight contact structure are equivalent notions, by a result of Hedden. In this note we analyze the relation between these two properties for non fibered links. A non fibered link (together with an incompressible Seifert surface) induces a natural partial open book by work of Feller, Lewark, and the second named author. We prove that strongly quasipositive links induce tight contact structures. We also show that, in contrast to the fibered case, the converse is not true, giving examples of links that are not strongly quasipositive but support tight contact structures.
\end{abstract}
\maketitle

\section{Introduction}
It is known by work of Hedden \cite{Hedden}, that for fibered links in $\mathbb{S}^3$, being strongly quasipositive and supporting a tight contact structure are equivalent notions. A fibered link $L$ induces an open book decomposition (of $\mathbb{S}^3)$, which in turn gives a contact structure. If we drop the fiberedness condition, the link does not induce an open book anymore.  However, given an incompressible Seifert surface for $L$, it might still admit \emph{product disks} (in the sense of Gabai), see \cite{Murasugisums} for the construction that is used. This allows us to define a partial open book in terms of these disks.
Partial open books, introduced by Honda, Kazez, and Mati\'{c} \cite{HKMPartial}, induce a contact manifold with convex boundary. In the case of a Seifert surface in $\mathbb{S}^3$ and the partial open book given by product disks, the contact manifold embeds in $\mathbb{S}^3$ smoothly, although perhaps not with a contact embedding. We show, in a similar way to Hedden's result, that if a link is strongly quasipositive its associated partial open book supports a tight contact structure.

Perhaps surprisingly, the converse is not true in general. We show that the knot $6_1$, which is not strongly quasipositive (it is not even quasipositive), has an associated partial open book that supports a tight contact structure. Note that this knot is \emph{nearly fibered} in the sense of \cite{Nearlyfibered} and \emph{almost fibered} in the sense of \cite{Fabiola-almost-fibered}, thus dropping the fiberedness condition ever so slightly produces a counterexample.
Building on this example, we give an infinite family of links with this properties. However, we also show that each of the surfaces we construct is (abstractly) diffeomorphic to a Seifert surface of a strongly quasipositive knot with the same product disks. 
\subsection{Structure of the paper.} 
In \Cref{sec:Prelims} we give some background on strongly quasipositive knots, partial open books, and the contact invariant for manifolds with convex boundary. Then in \Cref{sec:Results} we prove the results and end with some questions.
\subsection{Acknowledgments}
This note is the result of a research visit of the second author at the University of Glasgow. We thank the anonymous referee for useful comments.
IN was supported by EPSRC Studentship No. ~EP/W524359/1.
\section{Preliminaries} \label{sec:Prelims}

\subsection{Fiberedness and strong quasipositivity}
\begin{definition}[Fibered link]\label{definition:fibered_knot} 
A \emph{fibered link} in $\mathbb{S}^3$ is a link $L$ whose complement $\mathbb{S}^3 \setminus K$ fibers over the circle, i.e. there exists a fibration: 
$$  F \to \mathbb{S}^3 \setminus L \to \mathbb{S}^1, $$ where $F$ is a minimal genus Seifert surface for the link $L$.
\end{definition}

Fibered links have a natural open book $(S,h)$ associated to them, where $S$ is the fiber, and $h$ is the return map of the fibration. By the \emph{Giroux correspondence} \cite{Giroux}, this induces a contact structure on $\mathbb{S}^3$. Thus we can talk about contact properties of a fibered link when we talk about the properties of the associated contact structure. 

\begin{remark}
    Note that the minimal genus Seifert surface of a fibered link is unique up to isotopy, but this is not true in general if we drop the fiberedness condition .
\end{remark}
If a knot $K$ is non-fibered, the first thing to understand is how "far" $K$ is from being fibered. There are essentially two notions that measure the distance from fiberedness. One is given in terms of Heegaard Floer homology \cite{Baldwin-Sivek}, the other in terms of \emph{handle number} of a circular Heegaard splitting of the knot exterior $E(K)$ \cite{Fabiola-almost-fibered}. They hence produce two families of knots that are close to being fibered. 
\begin{definition}[Nearly fibered knot and almost fibered knot]\label{definition:nearly_fibered_knot}
    
    A knot $K$ is called \emph{nearly fibered} if the top grading in Heegaard Floer homology of $K$ is 2. 
    A knot $K$ is called \emph{almost fibered} if it is a  non-fibered knot K that possesses a circular Heegaard splitting $(E(K), F, S)$ such that:
    \begin{itemize}
        \item $F$ and $S$ are connected Seifert surfaces for $K$, 
        \item $ 1- \chi(S)$ realizes the circular width of $K$.
    \end{itemize}
\end{definition}
We will not really use the full theory of circular Heegard splittings in this paper, as it is beyond the scope of our work. Roughly  speaking, an almost fibered knot as per \cite{Fabiola-almost-fibered} has handle number 1, meaning that the circular Heegaard splitting has exactly one 1-handle. Fibered knots have circular Heegaard splitting with $F=S$ and no additional 1-handles are required.

\begin{remark}\label{almost_fibered_neq_nearly_fibered}
    We point out that the two families introduced above are distinct. It is known that all nearly fibered knots are almost fibered \cite{Nearlyfibered}, but not all almost fibered knots are nearly fibered \cite{Nearlyfibered}.
\end{remark}

The concepts in \Cref{definition:nearly_fibered_knot} are meant to give a measure of how far knots are from being fibered. We will not pursue this further, except to mention (see \Cref{remark:Stevedorealmostfibered}) that for our purpose there is simply a dichotomy rather than a hierarchy, i.e. Hedden's result is true for fibered knots, but not for non-fibered knots.

\begin{definition}[Quasipositive surface and strongly quasipositive links]\label{definition:quasipositive_surface}
    A surface with boundary $S \subset \mathbb{S}^3$ is \emph{quasipositive} if it is a subsurface of the fiber surface of a positive torus knot. A link that has a quasipositive Seifert surface is \emph{strongly quasipositive}.
\end{definition}
\begin{definition}[Murasugi sum]\label{definition:murasugi_sum}
    
    Let $\Sigma_1$ and $\Sigma_2$ be surfaces in $3$-manifolds $M_1$ and~$M_2$. 
    Let $n\geq1$ be an integer and let $P$ be the regular $2n$-gon whose
sides are circularly labeled $s_1, \dots, s_{2n}$. 
    For~$i = 1,2$, let $f_i$ be an embedding of $P$ into $\Sigma_i$ such that, for~$j = 1, \dots , n$, $f_1(s_{2j-1}) \subset \partial \Sigma_1$, $f_2(s_{2j}) \subset \partial \Sigma_2$, and $f_1(s_{2j})\subset \Sigma_1$ and $f_2(s_{2j-1})\subset \Sigma_2$ are properly embedded arcs in the respective surfaces.
    We identify $P$ with its image in~$\Sigma_1$,~$\Sigma_2$, and~$\Sigma$, and often suppress $f_i$ in the notation.
We assume that $\Sigma_i \cap B^3_i = P \subset \partial B^3_i \subset M_i$.
    Here, we arrange that a positive normal to $P\subset M_1$ points into $B_1^3$ and a positive normal to $P\subset M_2$ points out of~$B_2^3$.
    Furthermore, we arrange that $F$ exactly yields the chosen identification of~$P\subset M_i$, that is, $f_2\circ f_1^{-1}$ is the restriction of $F$ to the polygons.

    With this setup, we take $M\coloneqq M_1\# M_2$ and
    $\Sigma$ is the union of the copies of $\Sigma_1$ and $\Sigma_2$ in $M$ given by the canonical inclusions of $M_i^\circ$ into~$M$. In particular, after identifying  $\Sigma_1$ with their copies in~$M$, we have $\Sigma=\Sigma_1\cup \Sigma_2$ and $\Sigma_1\cap\Sigma_2=P$. See \Cref{figure:Murasugisum} for an illustration.

    The pair $(M, \Sigma)$ is called the \emph{Murasugi sum of $(M_1, \Sigma_1)$ and $(M_2, \Sigma_2)$ along the summing region $P$}. For simplicity, we also say that $\Sigma$ is the Murasugi sum of $\Sigma_1$ and~$\Sigma_2$, suppressing the $3$-manifolds in the notation. As we work with $M_1 = M_2 = \mathbb{S}^3$ throughout, this does not cause any confusion.
\end{definition}

\begin{figure}[htb] 
\centering
\begin{tikzpicture}
\node[anchor=south west,inner sep=0] at (0,0)
{\includegraphics[width=8cm]{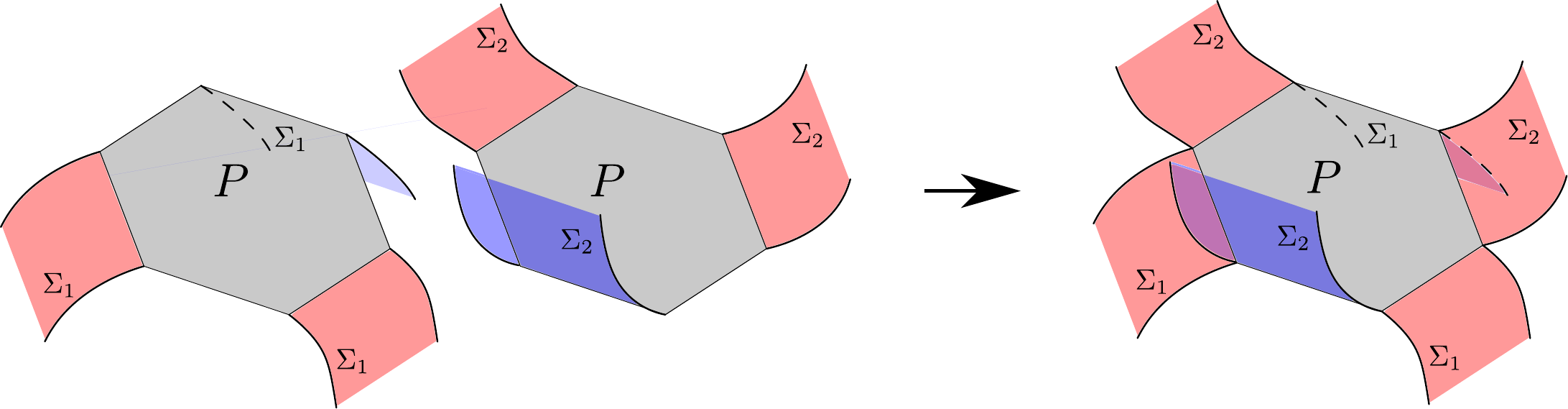}};
\end{tikzpicture}
\caption{A Murasugi sum of two surfaces $\Sigma_1$ and $\Sigma_2$ along $P$. Figure taken from \cite{Murasugisums}.}
\label{figure:Murasugisum}
\end{figure} 

  A result of Rudolph \cite{Rudolph} asserts that the result of a plumbing (more generally, a Murasugi sum) is strongly quasipositive if and only if the summands are. In particular, positive Hopf plumbing and deplumbing preserves strong quasipositivity. 
 
\subsection{Partial open books and contact manifolds with convex boundary}

\begin{definition}[Partial open book decompositions]\label{definition:partial_open_book_decomposition}
    A \emph{partial open book} is a triple $(S, P, h)$, where $S$ is a compact surface with boundary, $P$ is a subsurface of $S$ such that $S$ can be built from $S\setminus P$ by adding a finite number of $1$--handles, and $h: P \rightarrow S$ is an embedding.
\end{definition}

 By work of Honda, Kazez, and Mati\'{c} \cite{HKMPartial}, this determines a contact manifold with convex boundary. The manifold $(M,\xi)$ associated to $(S,P,h)$ is constructed as follows.
 We first build two handlebodies $H_1$ and $H_2$ from $S$ and $P$.
 $$H_1= \frac{S\times [-1,0]}{(x,t) \sim (x,s) \text{ for } x \in \partial S \text { and } t \in [-1,0]}$$
 $$H_2= \frac{P\times [0,1]}{(x,t) \sim (x,s) \text{ for } x \in \partial P \cap \partial S \text { and } t \in [0,1]}$$
 Then, we construct $M$ as $H_1 \cup_{h} H_2$, where the gluing via $h$ is given by $(x,0) \sim (i(x),0)$ and $(x,1) \sim (h(x),-1)$. Here $i:P \hookrightarrow S$ is the natural inclusion of $P$ in $S$ (as we are considering $P$ as a subsurface of $S$).
 The contact structure $\xi$ is defined to be the unique (up to isotopy) contact structure so that the restriction to both handlebodies is tight and the dividing curves on the convex boundary of $M$ are given by $\partial S \times \left(-\frac{1}{2}\right)$ and $\partial P \times \left(\frac{1}{2}\right)$.

\begin{remark} \label{remark:EmptyorfullP}
    If $P = S$ then we recover the usual notion of an \emph{abstract open book}, which gives rise to a closed $3$--manifold. If $P$ is empty, so is $H_2$, and we are left with the unique tight contact structure on $H_1$. We will thus always implicitly restrict ourselves to cases where $P \neq \varnothing$.
\end{remark}

 There is a relative analogue of Giroux correspondence \cite{HKMPartial}, so two partial open books support the same contact manifold if and only if they have a common positive stabilization (recall that a positive stabilization corresponds to plumbing a positive Hopf band).

\begin{definition}[Right--veering \cite{RV}]\label{definition:right--veering}
     Let $(S, \varphi)$ be an open book decomposition, and let $\alpha$ be an oriented properly
embedded arc with starting point $x$. We will adopt the convention that its image $\varphi(\alpha)$ is given the opposite orientation to $\alpha$. We then say that $\alpha$ is right--veering (with respect to $\varphi$) if $\varphi(\alpha)$ is isotopic
to $\alpha$ or, after isotoping $\alpha$ and $\varphi(\alpha)$ so that they intersect transversely with the fewest possible
number of intersections, $(\alpha'
(0), \varphi(\alpha)'
(1))$ define the orientation of $S$ at $x$. In this latter case we will
say that $\alpha$ is strictly right--veering. If $\alpha$ is not right--veering we say it is left-veering.
\end{definition}

\begin{remark}
    The isotopies in this definition are meant to be understood as isotopies of \emph{unoriented} arcs. Moreover, in \cite{RV} the arc $\varphi(\alpha)$ has the same orientation as $\alpha$, we follow the convention of later works, e.g. \cite{LegendrianSurgery}.
\end{remark}

If every arc in an open book is right--veering, we say that the open book itself is \emph{right--veering} (and analogously for left--veering). Similarly, if we have a partial open book $(S, P, h)$, if every arc $a$ in $P$ is right--veering (resp. left-veering) with respect to $h$ we say that the partial open book itself is \emph{right--veering} (resp. \emph{left--veering}).

 As in the closed case, Honda, Kazez, and Mati\'{c} prove that a contact manifold with convex boundary is overtwisted if an only if it has a supporting partial open book that is not right--veering \cite{HKMPartial}.

Moreover, the definition of the Heegaard Floer contact class (originally due to Oszv\'{a}th and Szab\'{o} \cite{OS} and reinterpreted by Honda, Kazez and Mati\'{c} \cite{HKMclosed}) can be modified to a class in sutured Floer homology \cite{HKMPartial}. As in the closed case, this class is zero for overtwisted contact structures. Therefore non--vanishing of this class is enough to guarantee tightness.

\section{Constructions and results} \label{sec:Results}

Recall from \Cref{sec:Prelims} that fibered links have a natural open book associated to them.  If the link is not fibered, we do not have such an open book. However, using Gabai's theory of product disks, we can define a partial open book, and consequently a contact structure on a manifold with convex boundary. Note that uniqueness of minimal genus Seifert surfaces is not guaranteed once we drop the assumption of the link being fibered. Hence we will need to declare our choice of surface $S$, not just the link $K$. We start with two preliminary definitions. First, we introduce the concept of a \emph{basis}.

\begin{definition}
    Let $S$ be a compact surface with boundary, and $S' \subset S$ a subsurface. A collection $\mathcal{B}$ of properly embedded arcs in $S$ is a \emph{basis} for $S'$ if $\mathcal{B}$ cuts $S'$ into a disk. 
\end{definition}

Now we recall the notion of \emph{product disks}, introduced by Gabai~\cite{Gabai_FoliationsII} in the context of sutured manifolds and foliations; and revisited in a setting more similar to ours in~\cite{Murasugisums}. For this reason we follow the definition in the latter.

\begin{definition} [\cite{Murasugisums}]
    A \emph{product disk} is a smooth map $\Delta\colon[0,1]\times[0,1]\to M$ with the following properties. The restriction to $[0,1]\times\{0\}$ and $[0,1]\times\{1\}$ defines proper arcs $a,a'$ in~$\Sigma$. The restrictions to $\{0\}\times [0,1]$ and $\{1\}\times [0,1]$ are constant, so in particular the two arcs have the same start and end points. The restriction to $(0,1)\times(0,1)$ is a diffeomorphism onto its image and the derivative at $(t,i)$ for $t\in(0,1)$ and $i\in\{0,1\}$ with respect to the second coordinate yields a tangent vector $v\in T_{\Delta(t,i)}M$ that points positively out of~$\Sigma$. Here, $v\in T_pM$ for $p\in\Sigma$ is said to \emph{point positively out of $\Sigma$}, if for a basis $(b_1,b_2)$ of $T_p\Sigma$ that defines the orientation of $T_p\Sigma$, we have that $(b_1,b_2,v)$ is a basis of $T_pM$ that defines the orientation of~$T_pM$. We say that a product disk \emph{witnesses $a \mapsto a'$}.
\end{definition}

For an example of a product disk, see Figure~\ref{figure:Pretzel_knot(3,3,1)}.

\begin{definition}[Partial Open Book associated to a link $K$ with an incompressible Seifert surface]\label{definition:partial_open_book_decomposition_of_sqp}
    Let $S$ be an incompressible  Seifert surface for $K$ in $\mathbb{S}^3$. Let $S'$ be the union of all arcs that have associated product disks, and for each arc $a$ with a product disk $D_a$, define $h(a)$ to be the arc such that $ D_a $ witnesses $a \mapsto h(a)$. Take $\mathcal{B}$ to be a basis of arcs of $S'$ such that the endpoints of every arc in $\mathcal{B}$ lie on $\partial S$. Let $P$ be a neighbourhood of $\mathcal{B}$, that is, the union of neighbourhoods of elements of $\mathcal{B}$. Then $(S, P, h)$ is the \emph{partial open book associated to $S$}.
\end{definition}
This construction, while natural, does not immediately give a well-defined notion. In the next lemma we show that it does.
\begin{lemma}\label{lemma:well_defined}
    Let $K$ be a link, $S$ an incompressible Seifert surface for $K$. Then $(S,P,h)$ is a well defined partial open-book associated to the pair $(S,K)$.
\end{lemma}

\begin{proof}

First, we show that for $a \neq b \in P$, and $D_a$ be a product disk witnessing $a \mapsto a'$, and $D_b$ a product disk witnessing $b \mapsto b'$, $a'$ and $b'$ are disjoint. Note that generically the intersection between $D_a$ and $D_b$ will be a $1$-manifold, i.e. a collection of circles and arcs. In fact, we can assume the intersection does not contain circles. Indeed, if it does, take an innermost intersection circle $C$ in $D_a$, that is, one that bounds a disk $D_C$ in $D_a$ with no further intersections with $D_b$. Then, replace $D_b$ with a disk $D'_{b}$ resulting from gluing the part of  $D_b$ lying outside of $C$ to (a small pushoff of) $D_C$. This removes the intersection circle $C$ without changing $b'$.

Now assume that there is an intersection point $x$ between $a'$ and $b'$. This must come from an intersection arc between $D_a$ and $D_b$, and crucially the other endpoint of this arc lies also in $a'$ and $b'$ since $a$ and $b$ are disjoint by construction. Take once again an innermost such arc $\gamma$ in $D_a$, which bounds a bigon in $D_a$ together with a segment $\delta_a$ of $a'$, which has no intersections with $D_b$. The arc $\gamma$ also bounds a bigon in $D_b$ together with a segment $\delta_b$ of $b'$. Gluing these two bigons together along $\gamma$ gives a disk with boundary $\delta_a \cup \delta_b$. This disk is embedded since the bigon in $D_a$ had no intersections with $D_b$. Then, incompressibility of $S$ implies that there is a disk in $S$ with boundary $\delta_a \cup \delta_b$. This disk is a bigon which shows that $x$ can be removed, and we are done.

Now, take an arc $a \subset S'$, and two product disks $D_a$ and $D_a'$ witnessing $a \mapsto a'$ and $a \mapsto a''$. Slightly isotope $D_a'$ to $D_b$ so that $\partial D_b$ is $b \cup b'$ with $b$ isotopic to, but disjoint from, $a$. The previous argument shows that $a'$ and $b'$ are disjoint. Now take $D_a \cup D_b$ (where here we also have to add the thin strip of isotopy between $a$ and $b$) to obtain a disk $D$. The boundary of this disk is $a' \cup b'$ (union small segments of the boundary of $S$), which are disjoint so we may assume that $D$ is embedded. But $S$ is incompressible, so $a' \cup b'$ is a disk, in other words, they are isotopic arcs, and $h$ is well defined up to isotopy.

Moreover, the product disks of a basis of $S'$ uniquely determine the product disks of all arcs in $S'$ because their complement in $S'$ is a disk.

Therefore, analogously to the fibered case, the contact manifold we obtain is independent of the choice of basis, and thus independent of the choice of $P$. Finally, since $\mathcal{B}$ is a collection of pairwise disjoint arcs, we can obtain $S$ from $S \setminus P$ by attaching $1$--handles.
\end{proof}

Note that, since the associated contact manifold $(M, \xi)$ is constructed using $h$, which is defined in terms of the product disks in $\mathbb{S}^3$, $M$ embeds in $\mathbb{S}^3$ (this embedding is smooth, but not necessarily contact).

We now prove that the partial open book associated to (a quasipositive surface of) a strongly quasipositive link supports a tight contact structure, using the following result.

\begin{lemma} {\cite[Proposition 5.1]{Murasugisums}} \label{lemma:RV}
    If $K$ is a strongly quasipositive link, the partial open book associated to a quasipositive Seifert surface is right--veering.
\end{lemma}

\begin{theorem}
    Strongly quasipositive links induce partial open books supporting tight contact structures.
\end{theorem}

\begin{proof}
Let $K$ be a strongly quasipositive link. Assume the open book $(S, P, h)$ associated to a quasipositive Seifert surface is overtwisted. By \cite[Proposition 4.1]{HKMPartial} there exists a partial open book $(S',P',h')$ for this contact manifold that is not right--veering. By the relative Giroux correspondence, they have a common stabilization $(S'',P'',h'')$. Recall that stabilization corresponds to plumbing a positive Hopf band, which preserves quasipositivity \cite{Rudolph}. Since $K$ is strongly quasipositive, $S''$ is quasipositive. But destabilization also preserves quasipositivity, so $S'$ is quasipositive, which contradicts \Cref{lemma:RV}. This concludes the proof. Note that, since $K$ is a link in $S^3$, all three open books (as well as the associated contact manifold $(M, \xi)$) embed in $S^3$, so we can indeed talk about quasipositivity of surfaces.
\end{proof}

\begin{remark}
    Note that, if an incompressible Seifert surface $S$ of a link has no product disks, then the associated contact structure is automatically tight, since $H_2$ is empty. There are many examples of such surfaces. For instance, take any unknotted annulus. If it has a non-zero number of twists it is incompressible. Moreover, it is strongly quasipositive if and only if it has a positive number of twists (a negative number of twists implies it is strongly quasinegative). However, it has a product disk if and only if it is a Hopf band (and therefore fibered). Note this gives many examples of links (both strongly quasipositive and not strongly quasipositive) that induce tight contact structures. However, we prefer to focus on surfaces that have nontrivial product disks.
\end{remark}

We now show that, unlike in the fibered case, the converse is not true.

\begin{proposition} \label{prop:counterexample}
    There exist a partial open book supporting a tight contact structure whose binding is not strongly quasipositive. 
\end{proposition}

\begin{proof}
    Consider the Stevedore knot $6_1$, which is not strongly quasipositive. We will show that its associated partial open book supports a tight contact structure by proving that the contact invariant is nonzero. We use the genus one Seifert surface $S$ in Figure~\ref{figure:Pretzel_knot(3,3,1)}. We know by \cite{Gabaipretzel} that it is the Murasugi sum of a positive Hopf link (which is fibered) and a band with $-4$ twists, which has no product disks. Then by \cite[Lemma 3.1]{Murasugisums}, the only product disks come from the Hopf band, which is right--veering. Now, the associated open book is $(S, P,h )$, where $P = N(a)$ and $h(a)$ is as shown ($a$ is the generator of the first homology of the Hopf link which is disjoint from the $-4$ twisted band), see Figure~\ref{figure:Pretzel_knot(3,3,1)}. Since $a$ and $h(a)$ have no interior intersections, it is clear that $c(S, P, h) \neq 0$. Indeed, a differential in Heegaard Floer homology involving a unique arc and its image will be a bigon, but if the arc is right--veering the target of this differential cannot be the contact class. This implies that the associated contact structure is tight \cite{HKMPartial}. Alternatively, note that Hopf plumbing does not change the contact structure. Since a band with $-4$ twists does not have any product disks, the contact structure is tight, see \Cref{remark:EmptyorfullP}.
\end{proof}
\begin{figure}[htb] 
\centering
\begin{tikzpicture}
\node[anchor=south west,inner sep=0] at (0,0)
{\includegraphics[width=6cm]{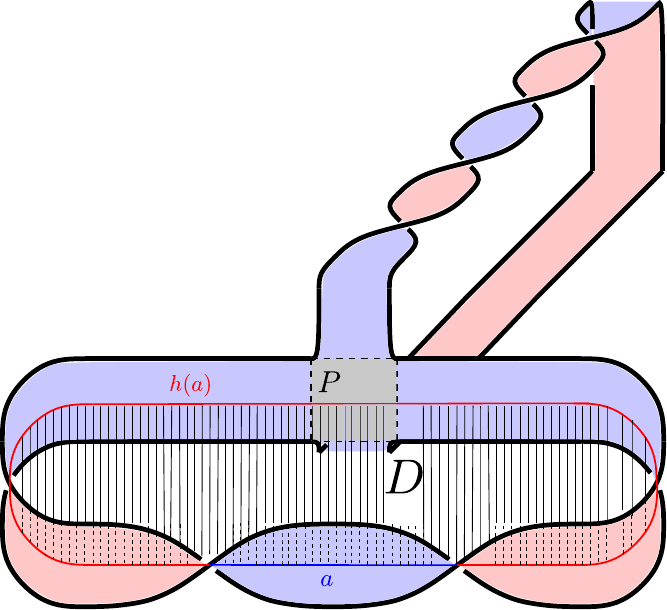}};
\end{tikzpicture}
\caption{The Stevedore knot $6_1$, obtained as a plumbing of a Hopf link and a band with $-4$ twists.}
\label{figure:Pretzel_knot(3,3,1)}
\end{figure} 
\begin{figure}[htb]
\centering
\label{figure:knot_on_surface}
\end{figure} 
\begin{remark} \label{remark:Stevedorealmostfibered}
    Note that the knot $6_1$ is \emph{nearly fibered} and hence \emph{almost fibered} (see Definition \ref{definition:nearly_fibered_knot}), so our counterexample is as strong a counterexample as one could hope for. In particular, dropping the fiberedness condition -- even if only slightly -- immediately implies that strong quasipositivity is not guaranteed by tightness of the supported (partial) open book.
\end{remark}
    Moreover, the technique of \Cref{prop:counterexample} can be used to produce infinitely many examples of links that support tight contact structures (indeed, contact structures with nonvanishing contact invariant) but are not strongly quasipositive. We illustrate this in the following examples.

\begin{Example}
    Take the knot $6_1$, and Murasugi sum it with any strongly quasipositive link. Since on the $3$--manifold level, and also the level of contact structures, Murasugi sums correspond to connected sums, the result is a tight contact manifold.
\end{Example}    
\begin{Example}
    Let $L$ be the pretzel link $(-3, n_2, \dots , n_k, 1)$, where some $n_i$ is negative, and $n_i + 1 \neq \pm 2$ for $i = 2, \dots k$. By \cite{Gabaipretzel}, this is the Murasugi sum of a positive Hopf link and $k-1$ bands, each with $n_i + 1$ twists for $k = 2, \dots k$. Then by \cite{Murasugisums}, the only product disks come from the positive Hopf link, and the argument from \Cref{prop:counterexample} applies. Moreover, since one of the coefficients is negative, the corresponding band is not strongly quasipositive, and thus by \cite{Rudolph} $L$ is not strongly quasipositive.
\end{Example}

\begin{remark}
    $6_1$ is an example of a slice knot which is not fibered and supports a tight structure. Note that nontrivial slice fibered knots can never support the tight contact structure --fibered knots supporting the tight contact structure are strongly quasipositive by Hedden \cite{Hedden}, and in this case $ 0 < g(K) = g_4(K)$.
\end{remark}

For every example above, the diffeomorphism type of the surface can be realised by a Seifert surface of a strongly quasipositive knot, such that the product disks induce the same map on the arc set. Indeed, if we replace every negative band in our examples with a positive band, the result is a strongly quasipositive link. Moreover, since the bands were not Hopf bands, again \cite{Murasugisums} shows that the only arc that admits a product disk is again $a$. Thus we ask the following question.

\begin{question} \label{Q1}
    Let $\Sigma$ be an incompressible Seifert surface whose associated partial open book induces a tight contact structure. Does the abstract partial open book also arise from a Seifert surface of a strongly quasipositive link?
\end{question}

\begin{remark}
    We did not introduce the concept of an abstract partial open book in this note. This can be understood in analogy to open books, see for example \cite{HKMPartial}. The main difference is that an \emph{abstract} partial open book encodes simply the diffeomorphism type of $S$, $P$, and $h$, while a partial open book as we have used it here depends on the specific way that $S$ is embedded in $\mathbb{S}^3$.
\end{remark}

As an avenue to answer in the negative Question \ref{Q1}, we propose investigating the following question.

\begin{question} \label{Q2}
    If $\Sigma$ is an incompressible Seifert surface of a strongly quasipositive link, does the associated partial open book have nontrivial contact invariant?
\end{question}

Note that an affirmative answer to \Cref{Q2} would imply a negative answer to \Cref{Q1}. Indeed, a neighbourhood of a Giroux torsion domain can be smoothly embedded in $S^3$, as topologically a Giroux torsion domain is just $T^2 \times I$. Given a partial open book of this contact manifold, we can take the resulting embedded surface in $S^3$. Assuming an affirmative answer to \Cref{Q2}, this cannot be a Seifert surface of a strongly quasipositive link, since it has vanishing contact invariant \cite{Girouxtorsion}. However, the associated contact structure is tight.

\bibliographystyle{myamsalpha} 
\bibliography{References}

\end{document}